\documentclass{amsart}
\usepackage{amssymb,enumerate,amsmath,color}

\def\Max{{\rm Max}}
\def\mf{\mathbf}
\def\f{\mathfrak}
\def\z{{\ldots}}

\newtheorem{thm}{ {Theorem}}[section]
\newtheorem{cor}[thm]{ {Corollary}}
\newtheorem{lem}[thm]{ {Lemma}}
\newtheorem{prop}[thm]{ {Proposition}}

\newtheorem{ex}[thm]{ {Example}}
\newtheorem{defn}[thm]{ {Definition}}
\newtheorem{oss}[thm]{Remark}
\begin{document}

\title[Invertibility of ideals in pr\"ufer extensions]
{Invertibility of ideals in pr\"ufer extensions}
\author{Carmelo Antonio Finocchiaro}
\address[Carmelo Antonio Finocchiaro]{Institute of Analysis and Number Theory\\  University of Technology\\
	8010 Graz, Kopernikusgasse 24 - Austria}
\email{finocchiaro@math.tugraz.at}
\author{Francesca Tartarone}
\address[Francesca Tartarone]{Dipartimento di Matematica\\ Universit\`{a}
degli studi Roma Tre\\ Largo San Leonardo Murialdo 1, 00146 Roma,
Italy} \email{tfrance@mat.uniroma3.it}

\thanks{The first named author was supported by a Post Doc Grant from the University of Technology of Graz (Austrian Science Fund (FWF): P 27816)}
\maketitle
\begin{abstract}
Using the general approach to invertibility for ideals in ring extensions  given by Knebush-Zhang in \cite{knzh}, we investigate about connections between faithfully flatness and invertibility for ideals in rings with zero divisors.
\end{abstract}
\section{Introduction}
It is well-known that in an integral domain  $D$ the invertible ideals (i.e. ideals $I$ for which there exists an ideal $J$ such that $IJ=D$) are exactly the nonzero finitely generated ideals that are also locally principal (\cite{bo}). In many cases, like for Noetherian domains, the hypothesis "finitely generated" can be omitted. For these domains the invertible ideals are exactly the nonzero locally principal  ideals (fact that is not true in general as it is shown by numerous examples).

In \cite{GV} the authors investigate about domains in which fathfully flat ideals are projective. We recall that in integral domains the notion of projective ideal is equivalent to the one of invertible ideal   (\cite[\S 2, Proposition 2.3]{knzh}). In general an invertible ideal is projective but not conversely. 

Since in integral domains locally principal ideals   coincide with the faithfully flat ideals (\cite[Theorem 8]{aa}), the question posed in \cite{GV} on the equivalence between fathfully flat and projective ideals leads to study domains in which nonzero locally principal ideals are invertible.

To this regard,  S. Bazzoni in \cite{bazzoni}  conjectured that  Pr\"ufer domains for which the equivalence    ``invertible ideal $\Leftrightarrow$ locally principal ideal" holds are exactly the ones with the finite character, i.e. each nonzero element of the domain belongs to finitely many maximal ideals.

 This conjecture was first proved in \cite{hmmt} and then it was extended to a larger class of domains  using the more general concepts of $t$-ideal and $t$-finite character (see, for instance, \cite{fipita}).



The notion of ideal, in this context,   is the one of fractional ideal; so it is related to the quotient field $K$ of the integral domain $D$ (we recall that a fractional ideal $I$ is a $D$-submodule of $K$ such that there exists $d \in D \setminus (0)$ for which $dI \subset D$).

Our aim in this paper is to study   the interplay among the concepts of invertible, faithfully flat and flat ideal in unitary rings $A$ with respect to   any their ring extension $B$. In this context we will generalize, for instance, the Bazzoni's conjecture about invertibility of ideals in Pr\"ufer domains, to the so-called Pr\"ufer extensions introduced by M. Knebush and D. Zhang in \cite{knzh}.

\section{flatness in ring extensions}
Let $A\subseteq B $ be a ring extension and let $S$ be  an $A$-submodule of $B$. Following \cite{knzh}, we say that $S$ is \textit{$B$-regular} if $SB=B$. Moreover, $S$ is \textit{$B$-invertible} if there exists an $A$-submodule $U$ of $B$ such that $SU=A$. By \cite[\S 2, Remark 1.10]{knzh}, if $S$ is $B$-invertible, then it is $B$-regular and  finitely generated, and $U$ is uniquely determined; precisely  $$U=[A:_BS]:=[A:S]:=\{x\in B:xS\subseteq A  \}.$$

More generally, given two $A$-submodules $S,T$ of $B$, we set $$[S:_BT]:=[S:T]:=\{x\in B:xT\subseteq S  \}.$$
It is well-known that a $B$-regular $A$-submodule $S$ of $B$  is $B$-invertible if and only if it is finitely generated and   locally principal  \cite[\S 2, Proposition 2.3]{knzh}. This fact generalizes the  characterization of invertible ideals in integral domains \cite[II \S 5, Theorem 4]{bo}.  From now on the term \emph{ideal} will always  mean an integral ideal.

\begin{prop}\label{inv-flat}
Let $A\subseteq B$ be a ring extension and let $\f a$ be an ideal of $A$. Then, the following conditions are equivalent. 
\begin{enumerate}[\rm(i)]
\item $\f a$ is $B$-invertible.
\item $\f a$ is $B$-regular, finitely generated and flat. 
\end{enumerate}
\end{prop}
\begin{proof}
(i)$\Longrightarrow$(ii). We need only to show that $\f a$ is flat, i.e., that for any ideal $\f b$ of $A$ the canonical map $f:\f a\otimes \f b\longrightarrow \f a$ is injective. By assumption, there are elements $\alpha_1,\z,\alpha_n\in \f a$ and $z_1,\z,z_n\in [A:\f a]$ such that $1=\displaystyle\sum_{i=1}^n\alpha_iz_i$. Take now elements $a_1,\z,a_m\in \f a,b_1,\z,b_m\in \f b$ such that $f(\displaystyle\sum_{j=1}^ma_j\otimes b_j):=\displaystyle\sum_{j=1}^ma_jb_j=0$. Then we have
$$
\displaystyle\sum_{j=1}^ma_j\otimes b_j=\displaystyle\sum_{j=1}^ma_j\left(\sum_{i=1}^n\alpha_iz_i\right)\otimes b_j=\sum_{i=1}^n\alpha_i\otimes\left(\sum_{j=1}^m(a_jz_i)b_j
\right)=\sum_{i=1}^n\alpha_i\otimes 0=0,
$$
and this proves that $f$ is injective. 

(ii)$\Longrightarrow$(i). It suffices to show that $\f a$ is locally principal. Assume there is a maximal ideal $\f m$ of $A$ such that the finitely generated and flat ideal  $\f aA_{\f m}$ of $A_{\f m}$ is not principal. By \cite[Lemma 2.1]{sava}, we have $\f aA_{\f m}\f mA_{\f m}=\f aA_{\f m}$, thus Nakayama's Lemma implies $\f aA_{\f m}=0$. On the other hand, the fact that $\f a$ is $B$-regular implies $B_{\f m}=\f aB_{\f m}=(\f aA_{\f m})B_{\f m}=0$, a contradiction. The proof is now complete. 
\end{proof}

Propositon \ref{inv-flat} generalizes  what is already known for ideals in integral domains (an ideal is invertible if and only if it is flat and finitely generated \cite{vasc}). Moreover, note that, with the same notation of Proposition \ref{inv-flat}, a $B$-invertible ideal of $A$ is also projective \cite[Chapter 2, Proposition 2.3]{knzh}.

\begin{oss}\label{faithful-remark}
Let $A$ be a ring. 
\begin{enumerate}[(a)]
\item If $\f a$ is a faithfully flat ideal of $A$, then $\f a$ is  locally principal. As a matter of fact, let $\f m$ be a maximal ideal of $A$. If $\f aA_{\f m}$ is not principal, then $\f aA_{\f m}=\f aA_{\f m}\f mA_{\f m}$, by  \cite [Lemma 2.1]{sava}. Thus 
$$
(0)=\f aA_{\f m}/\f a\f m A_{\f m}=\left(\f a/\f a\f m  \right)_{(A-\f m)}=A_{\f m}\otimes_A(\f a/\f a \f m)=A_{\f m}\otimes_A(A/\f m \otimes_A \f a)=(A_{\f m}\otimes_A A/\f m) \otimes_A \f a
$$
and, since $\f a$ is faithfully flat, we have that $A/\f m=A_{\f m}\otimes_A A/\f m=(0)$, which is a contradiction. 
\item A locally principal  ideal of a ring is not necessarily faithfully flat. For example, let $p\neq q$ be two fixed prime integers and let $A:=\mf Z/pq\mf Z$. We claim that  the principal ideal $\f i:= pA$ of $A$ is flat but not faithfully flat. Indeed, the equality $A=\f i \oplus qA$ shows that $\f i$ is projective and a fortiori flat. Moreover, $\f i=\f i^2$ and, since $\f i$ is a maximal ideal of $A$, it follows that $\f i$ is not faithfully flat by \cite[Chapter 1, Section 3, Proposition 1]{bo}. 
\end{enumerate}
\end{oss}

In  Remark \ref{faithful-remark} (b)  it is easy to check that the ideal $pA$ is not regular ($p$ itself is a zero divisor). 
In the next result we will see that if we assume the $B$-regularity of an ideal in an extension $B$ of $A$, then a locally principal ideal is faithfully flat. 

\begin{prop}\label{faithfully}
Let $A\subseteq B$ be a ring extension and let $\f a$ be a $B$-regular ideal of $A$. Then, the following conditions are equivalent. 
\begin{enumerate}[\rm (i)]
\item $\f a$ is faithfully flat.
\item $\f a$ is locally principal. 
\end{enumerate}
\end{prop}
\begin{proof}
(i)$\Longrightarrow$(ii) is Remark \ref{faithful-remark}(a). 

(ii)$\Longrightarrow$(i). Since $\f a $ is $B$-regular, for any maximal ideal $\f m$ of $A$, the principal ideal $\f aA_{\f m}$ is generated by an invertible element of $B_{A-\f m}$. Thus $\f aA_{\f m}$ is $B_{A-\f m}$-invertible and, a fortiori, a flat $A_{\f m}$-module. This proves that  $\f a$ is flat. Let $\f m$ be a maximal ideal of $A$. If $\f a\f m=\f a$, then $\f a\f m A_{\f m}=\f aA_{\f m}$ and, since $\f aA_{\f m}$  is $B_{A-\f m}$-invertible, $\f mA_{\f m}=A_{\f m}$, a contradiction. Thus $\f a\f m\neq \f a$, for any maximal ideal $\f m$ of $A$. By \cite[Chapter 1, Section 3, Proposition 1]{bo}, $\f a $ is faithfully flat. 
\end{proof}
\section{Ring extensions with the finite character}
\begin{defn}
Let $A \subseteq B$ be a ring extension. We say that the \emph{ring extension $A \subseteq B$ has the  finite character} if  any $B$-regular ideal $\f a$ of $A$ is contained in only finitely many maximal ideals (i.e., the set $V(\f a) \cap \Max(A)$ is finite where, as usual, $V(\f a)$ is the set of the prime ideals of $A$ containing $\f a$).
\end{defn}
It easily follows by definition that a ring extension $A\subseteq B$ has the finite character if and only if  any $B$-regular finitely generated ideal $\f a$ of $A$ is contained in only finitely many maximal ideals of $A$.
 
It is easy to see that if $A$ is an integral domain and $B$ is the quotient field of $A$, then $A \subseteq B$ has the finite character if and only if $A$  has the finite character. 

\bigskip
 Let $(\Omega, \leq )$ be a partially ordered set. We  denote by $\Max_\leq(\Omega)$ the (possibly empty) set of all maximal elements of $\Omega$. Two elements $x,y\in \Omega$ are said to be \emph{comaximal} if  there does not exist an element $m\in \Omega$ such that $x,y\leq m$. A subset $S$ of $\Omega$ is said to be \emph{comaximal} if any two elements $x,y\in S$ are comaximal. In the following Theorem \ref{dumi}, T. Dumitrescu and M. Zafrullah give conditions on  $\Omega$ under which the set of maximal elements greater than a fixed element $a \in \Omega$ is finite. This result will be useful in the following to characterize ring extensions $A \subseteq B$ with the finite character (see Corollary \ref{Bfinitechar}).
\begin{thm}\label{dumi}{\rm (\cite[Theorem 15]{duza})} Let $(\Omega,\leq)$ be a partially ordered set and let $\Gamma$ be a nonempty subset of $\Omega$. Assume that the following properties hold. 
\begin{enumerate}[\rm (a)]
\item For any $x\in \Omega$ there exists a maximal element $m\in \Omega$ such that $x\leq m$. 
\item If $a_1,a_2\in \Gamma$, $b\in \Omega$ and $a_1,a_2\leq b$, there is an element $a\in \Gamma$ such that $a_1,a_2\leq a\leq b$. 
\item If $b_1,b_2\in \Omega$ are comaximal elements, there are comaximal elements $a_1,a_2\in \Gamma$ such that $a_i\leq b_i$, for $i=1,2$. 
\end{enumerate}
Then, the following conditions are equivalent. 
\begin{enumerate}[\rm (i)]
\item For any $a\in \Gamma$, the set $\{x\in \Max_\leq(\Omega):x\geq a  \}$ is finite. 
\item For any $a\in \Gamma$, any comaximal subset of $\Gamma$ consisting of elements $\geq a$ is finite. 
\end{enumerate}
\end{thm}

\begin{cor}\label{Bfinitechar}
Let $A\subseteq B$ be a ring extension. Then, the following conditions are equivalent. 
\begin{enumerate}[\rm (i)]
\item The ring extension $A \subseteq B$ has the finite character.
\item For any finitely generated and $B$-regular ideal $\f a$ of $A$, each collection of mutually comaximal finitely generated and $B$-regular ideals containing $\f a$ is finite. 
\end{enumerate}
\end{cor}
\begin{proof}
Set
$$
\Omega:=\{\f a\subsetneq A:\f a \mbox{ is a }B\mbox{-regular ideal of }A   \}\quad \Gamma:=\{\f a\in \Omega: \f a \mbox{ is finitely generated}\}
$$
and order $\Omega$ by the inclusion $\subseteq$. Now, we claim that assumptions (a), (b) and (c) of Theorem \ref{dumi} are satisfied by $\Omega$ and $\Gamma$. Condition (a) follows immediately by the equality 
$
\Max_{\subseteq}(\Omega)=\Max(A)\cap \Omega
$, whose proof is straightforward. Now,  let $\f a_1,\f a_2\in \Gamma$ and let $\f b\in \Omega$ such that $\f a_1,\f a_2\subseteq \f b$. Thus $\f a:=\f a_1+\f a_2$ is a proper finitely generated $B$-regular ideal of $A$ such that $\f a_1,\f a_2\subseteq \f a \subseteq \f b$, and this proves condition (b). Now, take comaximal elements $\f b_1,\f b_2\in \Omega$. Then, the $B$-regular ideal $\f b_1+\f b_2$ cannot belong to $\Omega$, i.e., $\f b_1+\f b_2=A$. Pick elements $\beta_1\in \f b_1,\beta_2\in \f b_2$ such that $1=\beta_1+\beta_2$. On the other hand, since $\f b_1,\f b_2$ are $B$-regular, there are elements $\beta_{11},\z,\beta_{1n}\in \f b_1,\beta_{21},\z, \beta_{2m}\in \f b_2$ such that $$(\beta_{11},\z,\beta_{1n})B=(\beta_{21},\z, \beta_{2m})B=B.$$
Then the ideals $\f b_1':=(\beta_1,\beta_{11},\z,\beta_{1n} )A\subseteq \f b_1$, $\f b_2':=(\beta_2,\beta_{21},\z, \beta_{2m})A\subseteq \f b_2$ are finitely generated, comaximal and $B$-regular. This shows that also assumption (c) in Theorem \ref{dumi} is satisfied. The conclusion is now clear. 
\end{proof}
\begin{prop}\label{loc-fin-gen}
Let $A\subseteq B$ be a ring extension with the finite character. Then every $B$-regular and locally finitely generated ideal of $A$ is finitely generated. 
\end{prop}
\begin{proof}
Let $\f a$ be a $B$-regular and locally finitely generated ideal of $A$. By $B$-regularity, there exists a finitely generated ideal $\f a_0\subseteq \f a$ that is itself $B$-regular. By the finite character of the extension $A \subseteq B$, the set $\Max(A)\cap V(\f a_0)$ is finite, say 
$
\Max(A)\cap V(\f a_0)=\{\f m_1,\z,\f m_r, \f m_{r+1},\z, \f m_s   \}
$, where $\Max(A)\cap V(\f a)=\{\f m_1,\z,\f m_r \}$. For any $i=1,\z,r$, there is a finitely generated ideal $\f a_i\subseteq \f a$ such that $\f a A_{\f m_i}=\f a_iA_{\f m_i}$. For $j=r+1,\z,s$ pick an e
lement $a_i\in \f a-\f m_i$. Now consider the (finitely generated) ideal 
$
\f b:=\f a_0+\f a_1+\z+\f a_r+(a_{r+1},\z,a_s)A
$ of $A$. We claim that $\f a=\f b$. To prove this, it suffices to show the equality locally. If $\f m\in \Max(A)-V(\f a)$, then it is easy to infer that $\f aA_{\f m}=\f bA_{\f m}=A_{\f m}$. If $\f m=\f m_i$, for some $i=1,\z,r$, then 
$
\f a_iA_{\f m_i}\subseteq \f bA_{\f m_i}\subseteq \f aA_{\f m_i}=\f a_iA_{\f m_i}
$, that is $\f aA_{\f m_i}=\f bA_{\f m_i}=\f a_iA_{\f m_i}$. The proof is now complete. 
\end{proof}

The following result is an immediate consequence of  Propositions \ref{inv-flat}, \ref{faithfully} and \ref{loc-fin-gen}.

\begin{cor} \label{fin-char-inv}
Let $A\subseteq B$ be a ring extension with the finite character. Then, every $B$-regular and locally principal ideal of $A$ is $B$-invertible. 
\end{cor}

\begin{lem}\label{technical}
Let $A\subseteq B$ be a ring extension and let $\mathcal F$ be a collection of parwise comaximal $B$-invertible ideals of $A$ containing a fixed ideal $\f a$ of $A$. Consider the  ideal 
$$
\f i:=\f i_{\mathcal F}:=\left\{ a\in A:a\prod_{i=1}^n\f b_i\subseteq \f a, \mbox{ for some } \f b_1,\z,\f b_n\in \mathcal F  \right\}
$$
of $A$. Then, the following properties hold. 
\begin{enumerate}[\rm (a)]
\item If $\f i$ is finitely generated, then there is a finite subset $\mathcal G:=\{\f b_1,\z,\f b_n\}$ of $\mathcal F$ such that 
$$
\f i=\left\{ a\in A:a\prod_{i=1}^n\f b_i\subseteq \f a\right\}
$$
\item If $\f m$ is a maximal ideal of $A$, we have 
$$
\f iA_{\f m}=\begin{cases}
\f a A_{\f m} & \mbox{ if }\f m\nsupseteq \f b, \mbox{ for any }\f b\in\mathcal F\\
[\f a:\f b_0]A_{\f m} & \mbox{ if } \f b_0 \mbox{ is the unique ideal in } \mathcal F \mbox{ contained in }\f m 
\end{cases}
$$
\end{enumerate}

\end{lem}

\begin{proof}
The idea of the argument we are giving will follow, mutatis mutandis, the pattern of \cite[Lemma 1.10]{fipita}. 

If $\f i:=(x_1,\z,x_n)A$ and $\mathcal F_i \subseteq \mathcal F$ is a finite set of ideals such that $x_i\prod_{\f b\in \mathcal F_i}\f b\subseteq \f a$, then it suffices to take $\mathcal G:=\bigcup_{i=1}^n\mathcal F_i$. Thus (a) is clear. 

(b). Let $\f m$ be a maximal ideal of $A$. First, assume that none of the ideals in $\mathcal F$ is contained in $\f m$. Since clearly $\f a\subseteq \f i$, it is enough to prove that $\f iA_{\f m}\subseteq \f aA_{\f m}$. Consider a fraction $\frac{x}{s}\in \f iA_{\f m}$, with $x\in \f i$, let $s\in A-\f m$, and let $\f b_1,\z,\f b_n\in \mathcal F$ be such that $x\prod_{i=1}^n\f b_i\subseteq \f a$. For $i=1,\z, n$ choose elements $b_i\in \f b_i-\f m$. Then 
$$
\frac{x}{s}= \frac{x \prod_{i=1}^nb_i}{s\prod_{i=1}^nb_i}\in \f aA_{\f m}
$$
Now assume that there is an ideal $\f b_0\in\mathcal F$ such that $\f b_0\subseteq \f m$. The fact that the ideals in $\mathcal F$ are pairwise comaximal implies that $\f b\nsubseteq \f m$, for any $\f b\in \mathcal F-\{\f b_0\}$. Since $\f b_0$ is $B$-invertible and $\f a\subseteq \f b_0$, it follows easily that $[\f a:\f b_0]\subseteq \f i$. Let $x,s,\f b_1,\z,\f b_n$ be as before. If $\f b_0\neq \f b_i$, for any $i=1,\z,n$, then the same argument given above shows that $\frac{x}{s}\in \f aA_{\f m}\subseteq [\f a:\f b_0]A_{\f m}$. Assume now that $\f b_0=\f b_i$, for some $i$. Set $$p:=\begin{cases}
\prod_{j\neq i} b_j & \mbox{ if } \{\f b_1,\z,\f b_n  \}-\{\f b_0 \}\neq \emptyset\\
1 & \mbox{otherwise}
\end{cases}$$
where, as before, $ b_j\in \f b_j-\f m$, for $j\neq i$. Thus $xp\f b_0\subseteq \f a$, that is 
$$
\frac{x}{s}=\frac{xp}{sp}\in [\f a:\f b_0]A_{\f m}
$$
The proof is now complete. 
\end{proof}

\section{Pr\"ufer extensions}
In this section we use terminology and notation of \cite{knzh}. For the reader convenience, we give the basic definitions that are used in the following.

Let $R$ be a ring, $(\Gamma,\leq)$ be a (additive) totally ordered abelian group and $\infty\notin \Gamma$ be an element satisfying, by convention, $\gamma <\infty, \gamma\pm \infty:=\infty$, for any $\gamma \in \Gamma$. 
A Manis valuation on $R$ is a  map $v: R \rightarrow \Gamma \cup \{\infty\}$ satisfying the following conditions:

\begin{itemize}
\item $v(R)-\{\infty\}$ is a group.

\item $v(rs) = v(r) + v(s)$, for each $r, s \in R$.

\item $v(r+s) \geq \min\{v(r),v(s)\}$, for each $r,s\in R$.

\item $v(0) = \infty$. 

\end{itemize}

A \emph{Manis valuation subring} $A$ of $R$ is a subring of $R$ for which there exists a Manis valuation $v$ on $R$ with $A = A_v:=\{r\in R:v(r)\geq 0 \}$. Moreover,  $\f p:=\{r\in R:v(r)>0 \}$ is a prime ideal of $A_v$ and  $(A_v,\f p_v)$ is called  \emph{a $R$-Manis pair} (or \emph{a Manis pair in $R$}). 

For the reader convenience, we recall now the notion of generalized localization of a ring in an extension. Let $A\subseteq B$ be a ring extension, and let $\mathfrak p$ be a prime ideal of $A$. Let $j:B\longrightarrow B_{A\setminus \mathfrak p}$ denote the usual localization map. Then \emph{the generalized localization of $A$ at $\mathfrak p$ in $B$} is the subring $A_{[\f p]}:=j^{-1}(A_{\f p})$ of $B$ (\cite[p. 18]{knzh}). If $A$ is an integral domain with quotient field $K$ (in this case $K$ plays the role of $B$) and $\f p$ is a prime ideal of $A$, then the generalized localization of $A$ at $\mathfrak p$ in $K$ is the usual localization $A_{\f p}$.

\begin{defn} (\cite{knzh})
Let $A\subseteq B$ be a ring extension. We say that $A$ is weakly surjective in $B$ if $A_{\f m}=B_{A-\f m}$, for any maximal ideal $\f m$ of $A$ such that $\f mB\neq B$. We say that $A$ is \emph{Pr\"ufer in} $B$ or that $A\subseteq B$ \emph{is a Pr\"ufer extension} if $(A_{[\f m]}, \f m_{[\f m]})$ is a  Manis pair in $B$, for any maximal ideal $\f m$ of $A$. 
\end{defn}

We observe that in the case $A$ is an integral domain and $B$ is its quotient field, the ring extension $A \subseteq B$ is always weakly surjective.

We  recall that one of the many characterizations of Pr\"ufer domains says that  an integral domain is Pr\"ufer if and only if every nonzero finitely generated ideal is invertible (\cite{gi}).

The following theorem   is a  natural extension of the above result.  

\begin{thm}\label{prucharact}{\rm (\cite[Chapter 2, Theorem 2.1]{knzh})}
For a ring extension $A\subseteq B$, the following conditions are equivalent.
\begin{enumerate}[\rm (i)]
\item $A$ is Pr\"ufer in $B$.
\item $A$ is weakly surjective in $B$ and every finitely generated and $B$-regular ideal of $A$ is $B$-invertible. 
\end{enumerate}
\end{thm}

\begin{defn}
Let $A\subseteq B$ be a ring extension. We say that $A$ is \emph{almost Pr\"ufer in} $B$ (or that $A\subseteq B$ is \emph{an almost Pr\"ufer extension}) if every finitely generated $B$-regular ideal of $A$ is $B$-invertible. 
\end{defn}

As it is shown in the following Example \ref{almost prufer}, the   weakly surjectivity of an extension $A \subseteq B$ and the property that every finitely generated and $B$-regular ideal of $A$ is $B$-invertible are independent.

\begin{ex}\label{almost prufer}
Let $K$ be a field. Identify $K$ with a subring of $K^2$ via the diagonal embedding $x\mapsto(x,x)$. Then $K$ is almost Pr\"ufer on $K^2$, but it is not Pr\"ufer in $K^2$, since  $K$ is not weakly surjective in $K^2$, as it is easily seen. 
\end{ex}

Example \ref{almost prufer} shows that the class of Pr\"ufer ring extensions is properly contained in the class of almost Pr\"ufer extensions.

In Proposition \ref{faithfully} we show that, given a ring extension $A \subseteq B$, a $B$-regular ideal of $A$ is faithfully flat if and only if it is locally principal. Thus, we extend and prove Bazzoni's conjecture to almost Pr\"ufer ring extensions. We will see that in the following Theorem   the weakly surjective property of a ring extension $A \subseteq B$ is not needed to prove the statement. Thus, in this case, the notion of Pr\"ufer extension is too strong and we deal with almost Pr\"ufer extensions. 

\begin{thm}\label{bazzoni}
Let $A\subseteq B$ be an almost Pr\"ufer extension. Then, the following conditions are equivalent.
\begin{enumerate}[\rm (i)]
\item Every $B$-regular  and  locally principal ideal of $A$ is $B$-invertible.
\item The extension $A \subseteq B$ has the finite character. 
\end{enumerate}
\end{thm}
\begin{proof}
The implication (ii)$\Longrightarrow$(i) is Corollary  \ref{fin-char-inv} and holds without requiring that $A\subseteq B$ is almost Pr\"ufer. 

(i)$\Longrightarrow$(ii). Let $\f a$ be a finitely generated and $B$-regular ideal of $A$, and let $\mathcal F$ be a collection of mutually comaximal finitely generated and $B$-regular ideals of $A$  containing $\f a$. In view of Corollary \ref{Bfinitechar}, it suffices to prove that $\mathcal F$ is finite. Consider the ideal of $A$ defined in Lemma \ref{technical}:
$$
\f i_{\mathcal F}:=\left\{ x\in A: x\prod_{i=1}^n\f b_i\subseteq \f a, \mbox{ for some }\f b_1,\z,\f b_n\in\mathcal F
\right\}.
$$
First note that $\f i_{\mathcal F}$ is $B$-regular, since $\f a$ is $B$-regular and $\f a\subseteq \f i_{\mathcal F}$. Since $A\subseteq B$ is an almost Pr\"ufer extension, $\f a$ and all ideals in $\mathcal F$ are $B$-invertible. In particular, for any ideal $\f b\in \mathcal F$, the ideal $[\f a:\f b]$ is $B$-invertible. Keeping in mind Lemma \ref{technical}(b), it follows that the $B$-regular ideal $\f i_{\mathcal F}$ is locally principal and thus, by condition (i), it is $B$-invertible and, a fortiori, finitely generated. By Lemma \ref{technical}(a), we can pick finitely many ideals $\f b_1,\z,\f b_n\in \mathcal F$ such that $\f i_{\mathcal F}=\left\{x\in A: x\prod_{i=1}^n\f b_i\subseteq \f a \right\}$. If $\mathcal F=\{ \f b_1,\z,\f b_n \}$ we have done. Otherwise, take a ($B$-invertible) ideal $\f b\in\mathcal F-\{\f b_1,\z,\f b_n \}$.  Since $\f a[A:\f b]\f b=\f a$, we have $\f a[A:\f b]\subseteq \f i_{\mathcal F}$, and thus $\f a[A:\f b]\prod_{i=1}^n\f b_i\subseteq \f a$. Since $\f a$ is $B$-invertible, last inclusion implies $[A:\f b]\prod_{i=1}^n\f b_n\subseteq A$, and thus $\prod_{i=1}^n\f b_i\subseteq \f b$. Since $\f b$ is comaximal with each $\f b_i$, there is an ideal $\f c\subseteq \f b$ such that
$$
A=\prod_{i=1}^n(\f b+\f b_i)=\prod_{i=1}^n\f b_i+\f c\subseteq \f b+\f c\subseteq \f b\subseteq A.
$$
This shows that $\mathcal F-\{\f b_1,\z,\f b_n  \}\subseteq \{A \}$. Thus $\mathcal F$ is finite and, by Corollary \ref{Bfinitechar}, the result is completely proved. 
\end{proof}
Let $A$ be a commutative ring with identity and let $T(A)$ be the total quotient ring of $A$. An element $a\in A$ is said to be \emph{regular} if it is not a zero divisor, and an ideal of $A$ is said to be \emph{regular} if it contains a regular element. Recall that $A$ is a \emph{Pr\"ufer ring} if every finitely generated regular ideal of $A$ is invertible. Clearly, an ideal $\f a$ of $A$ is regular if and only if $\f a $ is $T(A)$-regular; $\f a$ is invertible if and only if $\f a$ is $T(A)$-invertible; and the extension $A\subseteq T(A)$ has the finite character if and only if every regular element of $A$ is contained in only finitely many maximal ideals of $A$ (i.e., $A$ has the finite character). Thus, by Theorem \ref{bazzoni}, we have
\begin{cor}
Let $A$ be a Pr\"ufer ring. Then every regular locally principal ideal of $A$ is invertible if and only if $A$ has the finite character. 
\end{cor}

\bigskip

The authors would like to thank the referee for her/his careful reading and for some substantial comments.

\end{document}